\documentclass[reqno]{amsart}
\usepackage{hyperref}

\def\R{{\mathbb{R}}}

\usepackage{amssymb}
\usepackage{xcolor}
\DeclareUnicodeCharacter{2264}{\leq}
\DeclareUnicodeCharacter{2266}{\leqq}

\theoremstyle{plain}
\newtheorem{theorem}{Theorem}
\newtheorem{proposition}{Proposition}
\newtheorem{definition}{Definition}
\newtheorem{lemma}{Lemma} 
\newtheorem{corollary}{Corollary}
\theoremstyle{remark}
\newtheorem{remark}{Remark}
\newtheorem{example}{Examples}

\title[Multiplier rules for Dini-derivatives]{Multiplier rules for Dini-derivatives in a topological vector space}
\author{Mohammed Bachir, Rongzhen Lyu}

\begin{document}

\date{\today} 
\address{Laboratoire SAMM 4543, Universit\'e Paris 1 Panth\'eon-Sorbonne, France}

\email{Mohammed.Bachir@univ-paris1.fr}
\email{Rongzhen.Lyu@etu.univ-paris1.fr}
\begin{abstract}
We provide new results of ﬁrst-order necessary conditions of optimality problem in the form of John’s theorem and in the form of Karush-Kuhn-Tucker’s theorem. We establish our result in a topological vector space for problems with inequality constraints and in a Banach space  for problems with equality and inequality constraints. Our contributions consist in the extension of the results known for  the Fr\'echet and Gateaux-differentiable functions as well as for the Clarke's subdifferential of Lipschitz functions  to the more general Dini-differentiable functions. As consequences, we extend the result of B.H. Pourciau in \cite[Theorem 6, p. 445]{Po} from the convexity to the {\it "Dini-pseudoconvexity"}.
\end{abstract}
\maketitle

{\bf 2020 Mathematics Subject Classiﬁcation:}  

\vskip5mm

\section{Introduction} \label{S1}

We consider the first-order necessary conditions for  optimality problems with inequality constraints in a  topological vector space and the first-order necessary conditions for  optimality problems with inequality and equality constraints in a Banach space. We will use the Dini-differentiability which is more general than the Fr\'echet or Gateaux-differentiability already used in the literature  to extend the recent results in \cite{Bl, Yi}. Let $E$ be a topological vector space and $\Omega$ be a nonempty open subset of $E$, and let $f_{i} : \Omega \rightarrow \mathbb{R}$ be functions, where $i \in \lbrace 0, ..., m \rbrace$. The modiﬁed lower Dini derivative of a function $f: E\to \R$ at $x\in \Omega$ in the direction $u\in E$ is defined as follows,
\[
 D^-_M f(x)(u) =\inf_{w\in E} \lbrace D^- f(x)(u+w)-D^-f(x)(w)\rbrace,
\]
where $D^- f(x)(u)$ denotes the lower Dini-derivative of $f$ at $x$ in the direction $u$:
\[
 D^- f(x)(u) := \liminf_{t\to 0^+} \frac{f(x+tu)-f(x)}{t}, 
\]
By convention we take $ D^- f(x)(0)=0$. We say that $f$ is $D^-_M$-differentiable at $x$ if  $D^-_M f(x)(u)$ and $D^- f(x)(u)$ are finite for every $u\in E$. The advantage of working with the modified lower Dini-differential $ D^-_M f(x)$ instead of  the classical lower Dini-differential $D^- f(x)$ is that the modified lower Dini-differential  is always a superlinear functional but it is not always the case for the classical lower Dini-differential, which is only positively homogeneous. However, the two notions coincide if and only if the classical lower Dini-differential $D^- f(x)$ is superlinear (Section \ref{S1}, for more details, we refer also to \cite{CRG, WY}). 

At first, we consider the following problem $(\mathcal{P}_1)$:
\begin{equation*}
(\mathcal{P}_1)
\left \{
\begin{array}
[c]{l}
\max f_0\\
x\in \Omega\\
\forall i \in \lbrace 1, ..., m \rbrace, f_{i}(x) \geq 0
\end{array}
\right. 
\end{equation*}
Our first main result (Theorem \ref{thmp} below) gives a first-order necessary conditions for the infinite-dimensional optimality problems  $(\mathcal{P}_1)$ by using the modified lower Dini-derivative which is more general than the Gateaux-differentiability. Its proof will be given at the end of  Section \ref{S2}, after some lemmas have been proven. 
\begin{theorem} \label{thmp}  Let $E$ be a topological vector space and $\hat{x}$ be a solution of problem $(\mathcal{P}_1)$. We assume that:

$(a)$   for all $j\in \lbrace i \in \lbrace 1, ..., m \rbrace : f_{i}(\hat{x}) < 0 \rbrace$, $f_j$ is lower semicontinuous at $\hat{x}$.

$(b)$ for all $i \in \lbrace 0, ..., m \rbrace$, $f_i$ is $D_M^-$-differentiable at $\hat{x}$.

 Then there exist $\lambda^{0}, ..., \lambda^{m} \in \mathbb{R}_{+}$ such that the following conditions hold:

$(i)$ $(\lambda^{0}, ..., \lambda^{m}) \neq (0, ..., 0)$.

$(ii)$ $\forall i \in \lbrace 1, ..., m \rbrace$, $\lambda^{i} f_{i}(\hat{x}) = 0$.

$(iii)$ $\sum_{i=0}^m  \lambda^{i} D^-_M f_{i}(\hat{x})(u) \leq 0$, for all $u\in E$.

If, in addition, we assume that the following assumption is fulﬁlled,

$(c)$ There exists $w\in E$ such that, $D^-_M f_i(\hat{x})(w) > 0$, whenever $i \in \lbrace 1, ..., m \rbrace$ satisfies $f_i(\hat{x}) = 0$,

then, we can take $\lambda^0=1$.
\end{theorem}
The above theorem applies in particular to any Lipschitz functions near $\hat{x}$. Notice that in the case when the functions $f_i$ are Gateaux-differentiable at $\hat{x}$, for all $i \in \lbrace 0, ..., m \rbrace$, then $D^-_M f_i (\hat{x})$ coincide with the classical Gateaux-differential $d_G f_i (\hat{x})$ which is a linear functional. Thus, the inequality in  part $(iii)$ of the above theorem becomes an equality and in this way, we get the results established in \cite{Bl, Yi}. 


 In \cite{Bl}, J. Blot considered the problem $(\mathcal{P}_1)$ in the case of $E=\R^n$ and extended the necessary conditions of optimality under the form of Fritz John’s conditions and under the form of Karush-Kuhn-Tucker’s conditions, from continuous Fr\'echet-differentiability to  Gateaux-differentiablitity.  There is another way to generalize the assumption of continuous Fr\'echet-diﬀerentiability by using locally Lipschitzian mappings and the Clarke's subdifferential or approximate subdifferential (see for instance \cite{Cf, JT1,JT2, Jourani}). However, it is known that in general the Clarke's subdifferential of a Gateaux-differentiable function does not always coincide with the Gateaux-derivative of the function. The advantage of the modified lower Dini-derivative in our result is that it applies to Lipschitz functions and moreover it coincides with the Gateaux-derivative when the function is Gateaux-differentiable (unlike Clark's subdifferential). Theorem \ref{thmp} announced above extend the recent results in \cite{Bl, Yi} by weakening the assumptions from Gâteaux-differentiable functions to the more general notion of the {\it modified lower Dini-differenitable} functions. It applies in particular (and in a unified way) to locally Lipschitz mappings, to lower semicontinuous concave functions, as well as to lower semicontinuous Gateaux-differentiable functions.  As a consequence, we also extend the result of B.H. Pourciau in \cite[Theorem 6, p. 445]{Po} from the convex framework to the more general pseudoconvex framework (see Corollary \ref{cor1}).  Extending our results to modified lower Dini-differentiability presents some difficulties due to the fact that unlike Gateaux-differentiability, the modified lower (resp. upper) Dini-differential is not linear in general but just superlinear (resp. sublinear). Lemma \ref{dMP} and a generalized Farkas lemma for sublinear functionals given in \cite[Corollary 2. p. 92]{Hl}, will make it possible to circumvent this difficulty.  
\vskip5mm
Our second main result (Theorem \ref{thmp2}) deals with a necessary condition of first order optimality for a problem with both inequality and equality constraints of the form
\begin{equation*}
(\mathcal{H}_1)
\left \{
\begin{array}
[c]{l}
\max f_0\\
x\in \Omega \\
\forall i \in \lbrace 1, ..., m \rbrace, f_{i}(x) \geq 0\\
h(x)=0
\end{array}
\right. 
\end{equation*}
where $E$ and $W$  are Banach spaces, $\Omega$ be an open subset of $E$ and  $f_i: \Omega \to \R$,  $i \in \lbrace 0, ..., m \rbrace$ and $h: \Omega \to W$ are  mappings. In our second result, we assume that the functions $f_i$ are Lipschitz in a neighborhood of the optimal solution and we will use the modified lower Dini-derivative and for the mapping $h$, we assume the continuous Fr\'echet-differentiability. Then, using the implicit function theorem, we will reduce the problem $(\mathcal{H}_1)$ to the problem $(\mathcal{P}_1)$, without equality constraints. In \cite{Bl, Yi}, it is assumed that the functions are Fr\'echet or Hadamard-differentiable and that $W$ is a finite-dimentional. 
\vskip5mm
This paper is organized as follows. In section 2, we recall the general definitions and notations that we will use later.  In section 3,  we give the proof  of the first main result Theorem \ref{thmp} at the end of the section after proving some lemmas and then we give Corollary \ref{cor1} as a consequence. In section 4, we state and prove our second main result Theorem \ref{thmp2}.

\section{Dini-Derivative: Definitions, notations and recall} \label{S1}
In all the paper $E$ denotes a topological vector space. A function $f: E \to \R$ is said to be lower (resp. upper) Dini-differentiable at $x \in E$ if $D^- f(x)(u)$ (resp. $D^+ f(x)(u)$) is  finite, for all directions $u \in E$, where 
\[
 D^- f(x)(u) := \liminf_{t\to 0^+} \frac{f(x+tu)-f(x)}{t}, 
\]
and
\[
D^+ f(x)(u) :=\limsup_{t\to 0^+} \frac{f(x+tu)-f(x)}{t}.
\]
We always assume that  $D^- f(x)(0)=D^+f(x)(0) = 0$. We see easily that for all $u\in E$, $-D^+ f(x)(u)=D^- (-f)(x)(u)$. We call  $D^- f(x)(\cdot)$ and $D^+ f(x)(\cdot)$, the lower and the upper Dini-differential of $f$ at $x$, respectively.  It is easy to see that on a normed space $(E,\|\cdot\|)$, any real-valued Lipschitz function in some neighborhood of $x\in E$ is both lower and upper Dini-differentiable at $x$ and moreover the map $u\mapsto D^- f(\hat{x})(u)$ (resp. the map $u\mapsto D^+ f(\hat{x})(u)$) is Lipschitz on $E$ with the same constant of Lipschitz of $f$ at $x$. 

Recall that  a function $p: E \rightarrow \mathbb{R}$ is sublinear if for all $x, y \in E$, $p(x+y) \leq p(x) + p(y)$, and for all $x \in E$, for all $\lambda \in \R^+$, $p(\lambda x) = \lambda p(x)$. A function $p: E \rightarrow \mathbb{R}$ is superlinear if for all $x, y \in E$, $p(x+y) \geq p(x) + p(y)$, and for all $x \in E$, for all $\lambda \in \R^+$, $p(\lambda x) = \lambda p(x)$. A function $p$ is sublinear if and only if $-p$ is superlinear. The lower (resp. upper) Dini-derivative of a function at some point is always positively homogeniuous but it is not superlinear (resp. sublinear) in general.  The modiﬁed lower Dini derivative of $f $ at $x$ in the direction $u$ is 
\[
 D^-_M f(x)(u) =\inf_{w\in E} \lbrace D^- f(x)(u+w)-D^-f(x)(w)\rbrace.
\]
In a similar way we define the modified upper Dini derivative of $f $ at $x$ in the direction $u$ as follows 
\[
 D^+_M f(x)(u) =\sup_{w\in E} \lbrace D^+ f(x)(u+w)-D^+f(x)(w)\rbrace.
\]
Notice that $$D^-_M f(x)(u)\leq D^-f(x)(u)\leq D^+f(x)(u) \leq D^+_M f(x)(u)$$ and $D^-_M f(x)(u)=-D^+_M (-f)(x)(u)$, for all $u\in E$. We refer to \cite{CRG, WY} for this notion of derivative and various other notions of directional derivatives, their comparisons as well as for related notions of subdifferentials.  Recall also the following classical directional derivatives which are always finite when $E$ is a normed space and $f$ is Lipschitz near $x$:

$\bullet$ The Clarke derivative of $f$ at $x$ in the direction $u$ is
$$f^\circ (x, u)=\limsup_{t\to 0^+; y\to x} \frac{f(y+tu)-f(y)}{t}.$$

$\bullet$ The Michel-Penot directional derivative of $f$ at $x$ in the direction $u$ is 
$$f^\diamond(x, u)=\sup_{y\in X} \limsup_{t\to 0^+} \frac{f(x+t(u+y))-f(x+ty)}{t}.$$
It is easy to see that for each $u\in E$, we have 
\begin{eqnarray*}
 D^+_M f(x)(u) \leq f^\diamond(x, u) \leq f^\circ (x, u).
\end{eqnarray*}
It is well know that (in a normed space) for a Lipschitz function near $x$, the functions $u\mapsto f^\diamond(x, u)$ and $u\mapsto f^\circ (x, u)$ are sublinear and globally Lipschitz  with the constant of Lipschitz equal to the constant of Lipschitz of $f$ at $x$. It is also known that a Lipschitz function $f$ is Gateaux differentiable at
$x$ if and only if $f^\diamond(x, \cdot)$  is a linear function (This is not always the case for the Clarke directional derivative). For more details concerning the Michel-Penot directional derivative and the upper Dini-derivative see \cite{MPe, Roc, Iof, Be}.

In this article we will assume that $E$ is a topological vector space and we will focus on the $D^-_M$-differentiability in the problem $(\mathcal{P}_1)$ of maximization and by symmetry, on the $D^+_M$-differentiability in the problem $(\mathcal{Q}_1)$ of minimization (see Section \ref{S2}).

\begin{definition} We say that a function $f: E \to \R$ is $D^-_M$-differentiable (resp. $D^+_M$-differentiable) at $x$ if  $D^-_M f(x)(u)$ and  $D^- f(x)(u)$ (resp. $D^+_M f(x)(u)$ and $D^+ f(x)(u)$) are finite for all $u\in E$. The map $D^-_M f(x)$ (resp. $D^+_M f(x)$) will be called the $D^-_M$-differential or the $D^-_M$-derivative (resp.  the $D^+_M$-differential or the $D^+_M$-derivative) of $f$ at $x$.
\end{definition}

The advantage of working with the modified lower (resp. upper) Dini-differentiability instead of  the classical lower (resp. upper) Dini-differentiability is that the modified lower (resp. upper) Dini-differential is  always a superlinear (resp. sublinear) functional but it is not always the case for the classical lower (resp. upper) Dini-differential. The two notions coincide if and only if the classical lower (resp. upper) Dini-differential is superlinear (resp. sublinear). 
\begin{proposition} \label{prop1} Let $E$ be a topological vector space and $f: E \to \R$ be a  function which is $D^-_M$-differentiable (resp. $D^+_M$-differentiable) at $x$. Then, $D^-_M f(x)$ is a superlinear functional on $E$ (resp.  $D^+_M f(x)$ is a sublinear functional on $E$). Moreover, we have that $D^- f(x)$ is superlinear  (resp. $D^+ f(x)$ is sublinear) if and only if $D^- f(x)=D^-_M f(x)$ (resp. $D^+ f(x)=D^+_M f(x)$).
\end{proposition}
\begin{proof} We give the result for $D^-_M f(x)$ the case of $D^+_M f(x)$ is shown in the same way. The positive homogeneity of $D^-_M f(x)$ is clear. Let us prove that  $D^-_M f(x)$ is superadditive. Let $u, v \in E$, then for every $\varepsilon>0$ there exists $w_\varepsilon \in E$ such that 
\begin{eqnarray*}
D^-_M f(x)(u+v) +\varepsilon &>& D^- f(x)(u+v+w_\varepsilon)-D^-f(x)(w_\varepsilon)\\
                                         &=& (D^- f(x)(u+v+w_\varepsilon)-D^- f(x)(v+w_\varepsilon))\\
&& + (D^-f(x)(v+w_\varepsilon) -D^-f(x)(w_\varepsilon)\\
                                         &\geq& D^-_M f(x)(u) +D^-_M f(x)(v).
\end{eqnarray*}
By sending $\varepsilon$ to $0$, we get that $D^-_M f(x)(u+v) \geq D^-_M f(x)(u) +D^-_M f(x)(v)$. Thus, $D^-_M f(x)$ is a superlinear functional on $E$. Now, if $D^- f(x)=D^-_M f(x)$, then $D^- f(x)$ is superlinear as $D^-_M f(x)$ is. To see the converse, suppose that $D^- f(x)$ is superlinear. Then, for every $w\in E$, $D^- f(x)(u+w) - D^- f(x)(w)\geq D^- f(x)(u)$. So, by taking the infinimum over $w\in E$, we get $D^-_M f(x)(u)\geq D^- f(x)(u)$. On the other hand, it is always true that $D^-_M f(x)(u)\leq D^- f(x)(u)$. Hence, $D^-_M f(x)(u)= D^- f(x)(u)$. 
\end{proof}

The following elementary proposition is easy to verify, so we mention it without proof.
\begin{proposition} \label{prop10} Suppose that $E$ is a normed space. Let $f: E \to \R$ be a Lipschitz function in some neighborhood of $x\in E$. Then, $f$ is  both $D^-_M$-differentiable and $D^+_M$-differentiable at $x$. Moreover the maps $u\mapsto D^-_M f(x)(u)$ and $u\mapsto D^+_M f(x)(u)$ are Lipschitz on $E$ with the same constant of Lipschitz of $f$ at $x$.
\end{proposition}

\begin{example} \label{Ex1}
We recall below some classical examples to which our result can be applied:

$1/.$ If $E$ is a normed space and $f$ is Lipschitz in a neighborhood of $x \in E$, then $f$ is both $D^-_M$-differentiable and $D^+_M$-differentiable at $x$ and the map $D^-_M f(x)$ is superlinear Lipschitz on $E$ and the map $D^+_M f(x)$  is sublinear Lipschitz on $E$ (see Proposition \ref{prop1} and proposition \ref{prop10}).

$2/.$  If $f$ is Gateaux-differentiable at $x\in E$, that is, the following limit exists 
$$d_G f(x)(u):=\lim_{t\to 0} \frac{f(x+tu)-f(x)}{t}, \hspace{1mm} \forall u\in E,$$
then clearly we have $D^-_M f(x)=D^+_M f(x)=D^- f(x)=D^+f(x)=d_G f(x)$ (a linear functional).

$3/.$ If $f$ is a real-valued convex  function on some convex neighborhood of $x \in E$, then $d^+ f(x)(u)$ (the "right hand " directional derivative of $f$ at $x$ in the direction $u$) exists for all $u\in E$, where 
$$d^+ f(x):=\lim_{t\to 0^+} \frac{f(x+tu)-f(x)}{t}.$$
In this case, we have $D^-_M f(x)(u)=D^+_M f(x)(u)=D^-f(x)(u)=D^+ f(x)(u)=d^+ f(x)(u)$  for all $u\in E$. Moreover, the map $u\mapsto d^+ f(x)(u)$ is a  sublinear functional on $E$ (see for instance \cite[Lemma 1.2]{Ph}). Replacing $f$ by $-f$, we see that if $f$ is a real-valued concave  function on some convex neighborhood of $\hat{x} \in E$, then $d^+ f(x)(u)$ exists and the map $u\mapsto d^+ f(x)(u)$ is a  superlinear functional on $E$.

$4/.$ Using $2/.$ and $3/.$ we see that $h:=f+g$ is $D^-_M$-differentiable at $x$ whenever $f$ is Gateaux-differentiable at $x$ and $g$ is a real-valued concave function. In this case we have $D^-_M h(x)=d^+ h(x)=d_G f(x)+d^+ g(x)$. Note that $h$ is in general, neither necessarily locally Lipschitz, nor Gateaux-differentiable at $x$ nor a concave function. 
\end{example}

\section{Optimization with inequality constraint} \label{S2}

This section is dedicated to the proof of Theorem \ref{thmp} which will be given at the end of the section. However, let us note that  since the maps $D^-_M f_i (\hat{x})$ are not linear in general but just superlinear functionals, we obtain the following symmetric version (Theorem \ref{thmp1} below) as an immediate consequence of Theorem \ref{thmp},  replacing $f_i$ by $-f_i$, $i \in \lbrace 0, ..., m \rbrace$ and  by changing {\it maximization} in the problem $\mathcal{P}_1$ by {\it minimization} in the problem $\mathcal{Q}_1$ and by changing  the modified lower Dini-differential by the modified upper Dini-differential which is a sublinear functional. Consider the following problem $\mathcal{Q}_1$ :

 \begin{equation*}
(\mathcal{Q}_1)
\left \{
\begin{array}
[c]{l}
\min f_0\\
x\in \Omega\\
\forall i \in \lbrace 1, ..., m \rbrace, f_{i}(x) \leq 0
\end{array}
\right. 
\end{equation*}

\begin{theorem} \label{thmp1}  Let $\hat{x}$ be a solution of the problem $\mathcal{Q}_1$. We assume that:

$(a)$  for all $j\in \lbrace i \in \lbrace 1, ..., m \rbrace : f_{i}(\hat{x}) < 0 \rbrace$, $f_j$ is upper semicontinuous at $\hat{x}$.

$(b)$ for all $i \in \lbrace 0, ..., m \rbrace$, $f_i$ is $D_M^+$-differentiable at $\hat{x}$.

 Then there exist $\lambda^{0}, ..., \lambda^{m} \in \mathbb{R}_{+}$ such that the following conditions hold:
\begin{itemize}
\item[(i)] $(\lambda^{0}, ..., \lambda^{m}) \neq (0, ..., 0)$.
\item[(ii)] $\forall i \in \lbrace 1, ..., m \rbrace$, $\lambda^{i} f_{i}(\hat{x}) = 0$
\item[(iii)] $\sum_{i=0}^m  \lambda^{i} D^+_M f_{i}(\hat{x})(u) \geq 0$, for all $u\in E$.
\end{itemize}
If, in addition, we assume that the following assumption is fulﬁlled,

$(c)$ There exists $w\in E$ such that, $D^+_M f_i(\hat{x})(w) < 0$, whenever $i \in \lbrace 1, ..., m \rbrace$ satisfies $f_i(\hat{x}) = 0$,

then, we can take $\lambda^0=1$.
\end{theorem}

 The proof of Theorem \ref{thmp} will be given at  the end of the paper after four lemmas. We start with the following useful lemma.

\begin{lemma} \label{dMP} Let  $\lbrace g_1,...,g_p\rbrace$ be a finite set of upper Dini-differentiable functions at $\hat{x}\in E$.  Then, for all $u\in E$,
$$\limsup_{t\to 0^+} \max_{k\in \lbrace 1,...,p\rbrace}\frac{g_k(\hat{x}+t u)-g_k(\hat{x})}{t} =  \max_{k\in \lbrace 1,...,p\rbrace} D^+g_k(\hat{x})(u).$$
By symmetry, if $\lbrace g_1,...,g_p\rbrace$ is a finite set of lower Dini-differentiable functions at $\hat{x}\in E$ we have, for all $u\in E$,
$$\liminf_{t\to 0^+} \min_{k\in \lbrace 1,...,p\rbrace}\frac{g_k(\hat{x}+t u)-g_k(\hat{x})}{t} =  \min_{k\in \lbrace 1,...,p\rbrace} D^-g_k(\hat{x})(u).$$
\end{lemma}
\begin{proof} From the definition of the upper Dini-differentiability, we have that for every $u\in E$, for  every $k\in \lbrace 1,...,p\rbrace$ and every $\varepsilon >0$, there exists $t_k > 0$ such that $\hat{x}+t_k u \in  \Omega_{\hat{x}}$ and
\begin{eqnarray}\label{eq1}
\sup_{0<t\leq t_k}\frac{g_k(\hat{x}+t u)-g_k(\hat{x})}{t} &<& D^+g_k(\hat{x})(u) + \varepsilon/2.
\end{eqnarray}

Set $s_0:=\min \lbrace t_k: k= 1,...,p \rbrace>0$. Then, there exists $k_0\in \lbrace 1,..., p\rbrace$ such that
\begin{eqnarray}\label{eq2}
\max_{k\in \lbrace 1,...,p\rbrace} \sup_{0<t\leq s_0} \frac{g_k(\hat{x}+t u)-g_k(\hat{x})}{t}  &=&  \sup_{0<t\leq s_0} \frac{g_{k_0}(\hat{x}+t u)- g_{k_0}(\hat{x})}{t} \nonumber\\
&\leq& \sup_{0<t \leq t_{k_0}} \frac{g_{k_0}(\hat{x}+t u)- g_{k_0}(\hat{x})}{t}\nonumber\\
&\leq& D^+g_{k_0}(\hat{x})(u) + \varepsilon/2 \nonumber\\
&\leq& \max_{k\in \lbrace 1,...,p\rbrace} D^+g_k(\hat{x})(u)+ \varepsilon/2.
\end{eqnarray}
It follows using $(\ref{eq1})$ and $(\ref{eq2})$ that 
\begin{eqnarray*}
\limsup_{t\to 0^+} \max_{k\in \lbrace 1,...,p\rbrace}\frac{g_k(\hat{x}+t u)-g_k(\hat{x})}{t} &=&  \inf_{s>0} \sup_{0<t\leq s} \max_{k\in \lbrace 1,...,p\rbrace} \frac{g_k(\hat{x}+t u)-g_k(\hat{x})}{t} \\ 
&\leq& \sup_{0<t\leq s_0} \max_{k\in \lbrace 1,...,p\rbrace} \frac{g_k(\hat{x}+t u)-g_k(\hat{x})}{t} \\
&=&\max_{k\in \lbrace 1,...,p\rbrace} \sup_{0<t\leq s_0} \frac{g_k(\hat{x}+t u)-g_k(\hat{x})}{t}\\
 &\leq& \max_{k\in \lbrace 1,...,p\rbrace} D^+g_k(\hat{x})(u)+ \varepsilon/2.
\end{eqnarray*}
Now, by sending $\varepsilon$ to $0$, we get 
\begin{eqnarray*}
\limsup_{t\to 0^+} \max_{k\in \lbrace 1,...,p\rbrace}\frac{g_k(\hat{x}+t u)-g_k(\hat{x})}{t} 
 &\leq& \max_{k\in \lbrace 1,...,p\rbrace} D^+g_k(\hat{x})(u).
\end{eqnarray*}
The converse inequality is trivial. Replacing $g_k$ by $-g_k$ in the above proof, we get by symmetry the second part of the Lemma.
\end{proof}

 In what follows (i.e. in Lemma \ref{lem1}, Lemma \ref{lemk} and Lemma \ref{lemd}), we assume that the functions $f_i$,  $i \in \lbrace 0, ..., m \rbrace$, satisfy the assumptions $(a)$ and $(b)$ of Theorem \ref{thmp} and we assume also that $$\lbrace i \in \lbrace 1, ..., m \rbrace : f_{i}(\hat{x}) = 0 \rbrace \neq \emptyset.$$ Up to a change of index, we denote $\lbrace 1, ..., l \rbrace = \lbrace i \in \lbrace 1, ..., m \rbrace : f_{i}(\hat{x}) = 0 \rbrace $.
The problem $\mathcal{P}_{1}$ can be reduced to the problem $\mathcal{P}_{2}$ as in the following lemma.
\begin{lemma} \label{lem1}  Let $\hat{x}$ be a solution of the problem $\mathcal{P}_1$, then there exists an open subset $\Omega_1$ of $\Omega$ containing $\hat{x}$ and such that  $\hat{x}$ is a solution of the following problem $\mathcal{P}_{2}$:

\begin{equation*}
(\mathcal{P}_2)
\left \{
\begin{array}
[c]{l}
\max f_0\\
x\in \Omega_1\\
\forall i \in \lbrace 1, ..., l \rbrace, f_{i}(x) \geq 0
\end{array}
\right. 
\end{equation*}
\end{lemma}
\begin{proof} When $1 \leq l < m$, we have that $f_{i}(\hat{x}) > 0$, for all $ i \in \lbrace l+1, ..., m \rbrace$. By the property of lower semicontinuity, there exists an open neighborhood $\Omega_{1} \subset \Omega$ of $\hat{x}$ such that $\Omega_1\subset \cap_{i \in \lbrace l+1, ..., m \rbrace} \lbrace x\in \Omega: f_{i}(x) > 0\rbrace$. 
When $l=m$, we take $\Omega_{1} = \Omega$. We then see that $\hat{x}$ is a solution of  $\mathcal{P}_{2}$.
\end{proof}

\begin{definition} \label{defA} $\forall p \in \lbrace 0, ..., l \rbrace$, we define the set:
\[
A_{p} = \lbrace u \in E : \forall i \in \lbrace p, ..., l \rbrace ,  D^-_M f_i (\hat{x})(u) > 0 \rbrace
\]
\end{definition}
 Since, $A_{i} \subset A_{i+1}$ for all $i\in \lbrace 0,...,l-1\rbrace$, we can have the following definition.

\begin{definition} \label{defk} Suppose that $A_{l} \neq \emptyset$ and define
\[
k : = \min \lbrace i \in \lbrace 0, ..., l \rbrace : A_{i} \neq \emptyset \rbrace
\]
\end{definition}

\begin{remark} By definition, we have that $A_{k} \neq \emptyset$ and $A_{k-1} = \emptyset$.
\end{remark}

\begin{lemma} \label{lemk} Suppose that $A_{l} \neq \emptyset$ (so that the Definition \ref{defk} has a sense). Then, we have  $ A_{0} = \emptyset $, in other words, we have $k\geq1$ (where $k$ is the integer number given by Definition \ref{defk}).
\end{lemma}
\begin{proof}  By contradiction, assume that $ A_{0} \neq \emptyset $. Since $D^-_M f_i(\hat{x})(u)\leq D^-f_i(\hat{x})(u)$ for all $u\in E$ and for all $i \in \lbrace 0, ..., l \rbrace$, then $$\emptyset \neq A_0\subset \lbrace u \in E : \forall i \in \lbrace 0, ..., l \rbrace ,  D^- f_i(\hat{x})(u) > 0 \rbrace.$$ Thus, there exists $w \in E$ such that for all $ i \in \lbrace 0, ..., l \rbrace$, $D^- f_{i}(\hat{x})(w) > 0$ and so $\min_{i\in \lbrace 0,...,l\rbrace} D^-f_i(\hat{x})(w)>0$.
Using Lemma \ref{dMP}, we have
$$\liminf_{t\to 0^+} \min_{i\in \lbrace 0,...,l\rbrace}\frac{f_i(\hat{x}+t w)-f_i(\hat{x})}{t} =  \min_{i\in \lbrace 0,...,l\rbrace} D^-f_i(\hat{x})(w)>0.$$
On the other hand, for $\varepsilon=\frac{1}{2}\min_{i\in \lbrace 0,...,l\rbrace} D^-f_i(\hat{x})(w)>0$ there exists $s_\varepsilon>0$ such that 
\begin{eqnarray*}
\inf_{0<t\leq s_\varepsilon} \min_{i\in \lbrace 0,...,l\rbrace}\frac{f_i(\hat{x}+t w)-f_i(\hat{x})}{t} +\varepsilon &>&\liminf_{t\to 0^+} \min_{i\in \lbrace 0,...,l\rbrace}\frac{f_i(\hat{x}+t w)-f_i(\hat{x})}{t}.
\end{eqnarray*}
It follows that 
\begin{eqnarray*}
\inf_{0<t\leq s_\varepsilon} \min_{i\in \lbrace 0,...,l\rbrace}\frac{f_i(\hat{x}+t w)-f_i(\hat{x})}{t} >\frac{1}{2}\min_{i\in \lbrace 0,...,l\rbrace} D^-f_i(\hat{x})(w)>0
\end{eqnarray*}
Hence, for all $i\in \lbrace 0,...,l\rbrace$, we have $f_i(\hat{x}+s_\varepsilon w) >f_i(\hat{x})$. In other words, we have $f_i(\hat{x}+s_\varepsilon w) >f_i(\hat{x})\geq 0$ for all $i\in \lbrace 1,...,l\rbrace$ and $f_0(\hat{x}+s_\varepsilon w) >f_0(\hat{x})$. This contradicts the assumption that $\hat{x}$ is an optimal solution of $\mathcal{P}_{2}$ (see Lemma \ref{lem1}). Therefore, $ A_{0} = \emptyset $.

\end{proof}

\begin{lemma} \label{lemd} Let $k\geq 1$ be the integer number given by Definition \ref{defk} and Lemma \ref{lemk}. Then, we have 
$$(v\in E,   D^-_M f_i(\hat{x})(v) \geq 0, \forall i \in \lbrace k, ..., l \rbrace) \Rightarrow ( D^-_M f_{k-1}(\hat{x})(v)\leq 0).$$
\end{lemma}
\begin{proof}
By contradiction, we assume that there exists $\upsilon \in E$ such that for all $i \in \{k, ..., l\}$, $D^-_M f_{i}(\hat{x})(\upsilon) \geq 0$ and $D^-_M f_{k-1}(\hat{x})(\upsilon) >  0$. Since $A_{k} \neq \emptyset$, there exists $\mu \in E$ such that for all $i \in \{k, ..., l\}$, $D^-_M f_{i}(\hat{x})(\mu) > 0$. Since $A_{k-1} = \emptyset$, $D^-_M f_{k-1}(\hat{x})(\mu) \leq 0$. If $D^-_M f_{k-1}(\hat{x})(\mu) < 0$, we choose $\epsilon$  such that $ 0 < \epsilon < - \frac{D^-_M f_{k-1}(\hat{x})(\upsilon)}{D^-_M f_{k-1}(\hat{x})(\mu)}$. If $D^-_M f_{k-1}(\hat{x})(\mu) = 0$, we arbitrarily take any $\epsilon > 0$. So in any case, we have $D^-_M f_{k-1}(\hat{x})(\upsilon) + \epsilon D^-_M f_{k-1}(\hat{x})(\mu) > 0$. Let $u_{\epsilon} := \upsilon + \epsilon \mu$, by the superlinearity of the modified lower Dini-differential, we have $D^-_M f_{k-1}(\hat{x})(u_{\epsilon}) \geq D^-_M f_{k-1}(\hat{x})(\upsilon) + \epsilon D^-_M f_{k-1}(\hat{x})(\mu) > 0$. On the other hand,  for all $i \in \{k, ..., l\}$, $D^-_M f_{i}(\hat{x})(u_{\epsilon}) \geq D^-_M f_{i}(\hat{x})(\upsilon) + \epsilon D^-_M f_{i}(\hat{x})(\mu) > 0$, this means that $u_{\epsilon} \in A_{k-1}$ which contradicts the fact that $A_{k-1} = \emptyset$. Hence, we have  $(v\in E,   D^-_M f_i (\hat{x})(v) \geq 0, \forall i \in \lbrace k, ..., l \rbrace) \Rightarrow ( D^-_M f_{k-1}(\hat{x})(v)\leq 0).$
\end{proof}

\begin{proof}[Proof of Theorem \ref{thmp}] If for all $i \in \lbrace 1, ..., m \rbrace$, $f_{i}(\hat{x}) > 0$, then parts $(i), (ii), (iii)$ of the theorem holds. Indeed, by the lower semicontinuity of $f_i$ for all $i \in \lbrace 1, ..., m \rbrace$, there exists an open neighborhood $\Omega_{\hat{x}}$ of $\hat{x}$ such that 
$$\Omega_{\hat{x}}\subset \lbrace x\in \Omega: f_{i}(x) \geq 0\rbrace,$$
and so $\hat{x}$ maximizes $f_{0}$ on the open subset $\Omega_{\hat{x}}\subset \Omega$. Then, for all $u \in E$ and for all $t>0$ such that $\hat{x}+tu\in \Omega_{\hat{x}}$, we have that $f_{0}(\hat{x}) \geq f_{0}(\hat{x}+tu)$, which gives that $D^-  f_{0}(\hat{x}) \leq 0$. In particular we have $D^-_M f_{0}(\hat{x})\leq D^-  f_{0}(\hat{x}) \leq 0$. We take $\lambda^{0} = 1$ and $\forall i \in \lbrace 1, ..., m \rbrace$,  $\lambda^{i} = 0$. Therefore, $(i), (ii), (iii)$ hold with these settings of parameters. Thus, in the following, we assume that $\lbrace i \in \lbrace 1, ..., m \rbrace : f_{i}(\hat{x}) = 0 \rbrace \neq \emptyset$. Up to a change of index, we denote $\lbrace 1, ..., l \rbrace = \lbrace i \in \lbrace 1, ..., m \rbrace : f_{i}(\hat{x}) = 0 \rbrace $. Using Lemma \ref{lem1}, there exists an open subset $\Omega_1$ of $\Omega$ containing $\hat{x}$ and such that  $\hat{x}$ is a solution of the following problem
\begin{equation*}
(\mathcal{P}_2)
\left \{
\begin{array}
[c]{l}
\max f_0\\
x\in \Omega_1\\
\forall i \in \lbrace 1, ..., l \rbrace, f_{i}(x) \geq 0
\end{array}
\right. 
\end{equation*}
We have to treat two cases on the set $A_l$ given in Definition \ref{defA}:

{\it Case1:} If  $A_{l} = \emptyset$, there does not exist any $\omega \in E$ such that $ D^-_M f_{l} (\hat{x})(\omega) > 0$, therefore, $D^-_M f_{l}(\hat{x})(u) \leq 0$ for all $u\in E$. We take $\lambda^{l} = 1$, and $\lambda^{i} = 0$, for $i \in \{0, ..., m\} \setminus \{l\}$. In this case,  parts $(i), (ii), (iii)$ of  Theorem \ref{thmp} hold.

{\it Case2:} If $A_{l} \neq \emptyset$, then by Lemma \ref{lemk}, we have that $k\geq 1$. Using Lemma \ref{lemd} we have $$(v\in E, - D^-_M f_{i}(\hat{x})(v) \leq 0, \forall i \in \lbrace k, ..., l \rbrace) \Rightarrow (- D^-_M f_{k-1}(\hat{x})(v) \geq 0).$$ Moreover, the slater condition (the consistency) is satisfied by the definition of $k$ ($A_k\neq \emptyset$), that is, there exists some $w\in E$ such that $ D^-_M f_i (\hat{x})(w) > 0$, for all $i \in \lbrace k, ..., l \rbrace$. Let us set $K:=\R^{l-k+1}_+$, the positive cone of $\R^{l-k+1}$ and 
\begin{eqnarray*}
S:=(-D^-_M f_{k}(\hat{x}),...,-D^-_M f_{l} (\hat{x})): E &\to& \R^{l-k+1}\\
f:=-D^-_M f_{k-1}(\hat{x}): E &\to& \R.
\end{eqnarray*}
 Since the maps $-D^-_M f_i (\hat{x})$, $i=k-1,...l$, are sublinear functionals (hence convex functions), by using the generalized Farkas's lemma in \cite[Corollary 2. p. 92]{Hl}, there exists $\mu^k, ..., \mu^l \in \mathbb{R}_{+}$, such that $D^-_M f_{k-1}(\hat{x}) + \sum_{i=k}^l \mu^i D^-_M f_i (\hat{x}) \leq 0$. We set: 
\begin{equation*}
\lambda^i : =
\begin{cases}
 1,   & i = k-1\\
 \mu^i,   & i \in \{ k, ..., l \}\\
 0,   & i \in \{0, …, m \}\setminus \{k-1, …,l \}.
\end{cases}
\end{equation*}
Therefore, we have proven that if $A_{l} \neq \emptyset$, parts $(i), (ii), (iii)$ of Theorem \ref{thmp} also hold (recall that $f_i(\hat{x})=0$, for all $i\in \lbrace 1,...,l\rbrace$). 

Finally, if moreover we assume $(c)$, that is, there exists $w\in E$ such that,  $D^-_M f_i(\hat{x})(w) > 0$ whenever $i \in \lbrace 1, ..., m \rbrace$ satisfies $f_i(\hat{x}) = 0$, then we can take $\lambda^0=1$. Indeed,  suppose by contradiction that $\lambda^0=0$ then from part $(i)$ we have $(\lambda^{1}, ..., \lambda^{m}) \neq (0, ..., 0)$. Since, $\lambda^i=0$ whenever $i\in \lbrace 1, ..., m \rbrace$ and $f_i(\hat{x})\neq 0$ by part $(ii)$, we get that $\sum_{i=1}^m  \lambda^{i} D^-_M f_{i}(\hat{x})(w) >0$, which contradicts the part $(iii)$. Hence, $\lambda^0\neq 0$. Dividing  if necessary the multipliers $(\lambda^i)_{0\leq i \leq m}$ by $\lambda^0\neq 0$, we can assume that $\lambda^0=1$. This ends the proof.
\end{proof}

As we mentioned in the introduction, our result generalizes those in  \cite{Bl, Yi}. As a consequence, we also extend  the result of B.H. Pourciau in \cite[Theorem 6, p. 445]{Po}. Indeed,  on the one hand we deal with any topological vector space instead of a finite dimensional space and on the other hand, we deal with a class of functions more general than the class of convex functions.  Let $E$ be a topological vector space, $\Omega$ be an open  subset of $E$ and  $F:\Omega \to \R^n$ be a function defined by $F(x)=(f_1(x),...,f_n(x))$ for all $x\in \Omega$. We say that $F$ is $D^+_M$-differentiable at some point $\hat{x}\in \Omega$, if $f_i$ is $D^+_M$-differentiable at $\hat{x}$ for all $i \in \lbrace 1,...,n\rbrace$. In this case, we denote
$$D^+_M F(\hat{x}):=(D^+_M f_1(\hat{x}),...,D^+_M f_n(\hat{x})).$$
We say that $F$ is {\it $D^+_M$-invex} at $\hat{x}\in \Omega$, if $F$ is $D^+_M$-differentiable at $\hat{x}$ and satisfies: $ \forall x \in \Omega, \hspace{1mm} \exists \eta=\eta(x,\hat{x})\in E:$
 $$ \forall w\in (\R^+)^n, \langle w, D^+_M F(\hat{x})(\eta) \rangle \leq\langle w, F(x)- F(\hat{x})\rangle.$$
Real-valued invex functions were introduced by Hanson \cite{Ha} as a generalization of convex functions. The notion of $K$-invexity with respect a convex cone $K$ for a vector-valued  function, was given in \cite{Cr}. We refer to \cite{YS, Re} for different type of $K$-invexity and the relationship between them. The definition of {\it $D^+_M$-invexity} that we have given above corresponds to the so-called {\it "restricted K-invex in the limit at $\hat{x}$"} in \cite{YS, Re} (here, we have taken $K=(\R^+)^n$) replacing the {\it generalised directional derivative} of Clarke \cite{Cf} with the modified upper Dini-derivative.
We say that $F$ is {\it $D^+_M$-pseudoconvex} at $\hat{x}\in \Omega$, if $F$ is $D^+_M$-differentiable at $\hat{x}$ and satisfies: $ \forall x \in \Omega \hspace{1mm} \exists \eta=\eta(x,\hat{x})\in E:$
 $$ \forall w\in (\R^+)^n, \langle w, D^+_M F(\hat{x})(\eta) \rangle \geq 0  \Longrightarrow \langle w, F(x)- F(\hat{x})\rangle \geq 0.$$
Clearly, if $F$ is {\it $D^+_M$-invex} at $\hat{x}$, then it is {\it $D^+_M$-pseudoconvex} at $\hat{x}$ (for a more general concept of pseudoconvex functions we refer to \cite{HP, Iv} and the references therein).

We give the following extension of \cite[Theorem 6, p. 445]{Po} as an immediate consequence of Theorem \ref{thmp1}.
\begin{corollary} \label{cor1} Let $E$ be a topological vector space, $\Omega$ be a nonempty open  subset of $E$ and $f_i: \Omega \to \R$ be a  function  for all $i \in \lbrace 0, ..., m \rbrace$. Let $\hat{x}$ be a solution of the problem $\mathcal{Q}_1$. Assume that $F=(f_0, f_1,...,f_m)$ is $D^+_M$-pseudoconvex at $\hat{x}$  and that for all $j\in \lbrace i \in \lbrace 1, ..., m \rbrace: f_{i}(\hat{x}) < 0 \rbrace$, $f_{j}$ is upper semicontinuous at $\hat{x}$. Then, there exists some nonzero $(\lambda^0,...,\lambda^m)\in \R^{m+1}$ satisfying

$(a)$ $\lambda^i\geq 0$ for $i=0,...,m$,

$(b)$ $\lambda^i f_i (\hat{x})=0$ for $i=1,...,m$ and

$(c)$  if $f=\sum_{i=0}^m \lambda^i f_i$ and $x\in \Omega$, then $\lambda^0 f_0(\hat{x})=f(\hat{x}) \leq f(x)$.
\end{corollary}
\begin{proof}   We apply Theorem \ref{thmp1} and we get some nonzero $(\lambda^0,...,\lambda^m)\in \R^{m+1}$ satisfying $(a)$, $(b)$ and $\sum_{i=0}^m  \lambda^{i} D^+_M f_{i}(\hat{x})(u) \geq 0$, for all $u\in E$.  Set $f=\sum_{i=0}^m \lambda^i f_i$. Since, $F$ is $D^+_M$-pseudoconvex at $\hat{x}$, we have  for all $ x\in \Omega$, there exists $\eta=\eta(x, \hat{x})\in E$ such that
 $$ \forall w\in (\R^+)^n, \langle w, D^+_M F(\hat{x})(\eta) \rangle \geq 0  \Longrightarrow \langle w, F(x)- F(\hat{x})\rangle \geq 0.$$
In particular, with $w=(\lambda^0,...,\lambda^m)$, since  $\sum_{i=0}^m \lambda^i D^+_M f_i(\hat{x})(\eta)\geq 0$, we get that $f(x)-f(\hat{x})\geq 0$. On the other hand, thanks to $(b)$, we see that $\lambda^0 f_0(\hat{x})=f(\hat{x})$. Hence, we give part $(c)$.
\end{proof}


\section{Optimization with inequality and equality constraint} 
We give below our second main result. In comparison with existing results (see for instance \cite{Bl, Yi}) we weaken assumptions of Gateaux-differentiability of the functions in the inequality constraints, by assuming that the functions are Lipschitz in a neighborhood of the optimal solution and by using Dini-differentiability instead of Clarke's subdifferential. Let $(E,\|\cdot\|_E)$ and $(W,\|\cdot\|_W)$  be Banach spaces, $\Omega$ be an open subset of $E$.  Let  $f_i: \Omega \to \R$,  $i \in \lbrace 0, ..., m \rbrace$ and $h: \Omega \to W$ be  mappings. Consider the following problem:
\begin{equation*}
(\mathcal{H}_1)
\left \{
\begin{array}
[c]{l}
\max f_0\\
x\in \Omega \\
\forall i \in \lbrace 1, ..., m \rbrace, f_{i}(x) \geq 0\\
h(x)=0
\end{array}
\right. 
\end{equation*}

\begin{theorem} \label{thmp2}  Let $\hat{x}$ be a solution of the problem $\mathcal{H}_1$. We assume that:


$(a)$ for all $i \in \lbrace 0, ..., m \rbrace$, $f_i$ is Lipschitz in a neighborhood of $\hat{x}$ (in particular $f_i$ is $D_M^-$-differentiable at $\hat{x}$).

$(b)$ $h$ is  Fr\'echet-differentiable in a neighborhood of  $\hat{x}$, the Fr\'echet-differential $d_F h(\cdot)$ is continuous at $\hat{x}$ and $\textnormal{Im} d_F h(\hat{x})$ is closed.

$(c)$ $\textnormal{Ker}(d_F h(\hat{x}))$ is a complemented subspace of $E$, that is, there exists a closed subspace $E_1$ of $E$ such that $E=\textnormal{Ker}(d_F h(\hat{x}))\oplus E_1$.
\vskip5mm

 Then there exist $\lambda^{0}, ..., \lambda^{m} \in \mathbb{R}_{+}$ and $w^*_0\in W^*$ such that the following conditions hold:
\begin{itemize}
\item[(i)] $(\lambda^{0}, ..., \lambda^{m},w^*_0) \neq (0, ..., 0, 0_{W^*})$.
\item[(ii)] $\forall i \in \lbrace 1, ..., m \rbrace$, $\lambda^{i} f_{i}(\hat{x}) = 0$.
\item[(iii)] $\sum_{i=0}^m  \lambda^{i} D^-_M f_{i}(\hat{x})(u) + \langle w^*_0 , d_F h(\hat{x})(u)\rangle \leq 0$, for all $u\in E$.
\end{itemize}
where, $(\lambda^{0}, ..., \lambda^{m},w^*_0)$ can be chosen as follows

$\bullet$ $(\lambda_0,..., \lambda^{m},w^*_0)=(0,...,0,w^*_0)$, with $w^*_0\neq 0$ if $d_F h(\hat{x})$ is not onto.

$\bullet$ $(\lambda_0,..., \lambda^{m})\neq (0,...,0)$, if $d_F h(\hat{x})$ is onto. If moreover, there exists $w\in \textnormal{Ker}(d_F h(\hat{x}))$ such that, $D^-_M f_j(\hat{x})(w) > 0$, whenever $j\in \lbrace i \in \lbrace 1, ..., m \rbrace : f_i(\hat{x}) = 0 \rbrace$, then we can chose $\lambda_0=1$.
\end{theorem}

\begin{proof} We consider two cases. If $\textnormal{Im} d_F h(\hat{x})\neq W$, then since $\textnormal{Im} d_F h(\hat{x})$ is closed, then by the Hahn-Banach theorem, there exists some $w^*\in W^*\setminus \lbrace 0\rbrace$ such that $w^*\circ d_F h(\hat{x})=0$. Thus, the theorem work with $(\lambda_0,..., \lambda^{m})=(0,...,0)$ and $w^*\neq 0$.

Now, suppose that $\textnormal{Im} d_F h(\hat{x})= W$. In this case, the restriction $d_F h(\hat{x})_{|E_1}: E_1 \to W$ is an isomorphism.  Let $(\hat{a},\hat{b})\in \textnormal{Ker}(d_F h(\hat{x}))\times E_1$ such that $\hat{x}=\hat{a}+\hat{b}$. Since, $h(\hat{x})=0$, then  by the implicit function theorem (see for example \cite{Lang}), there exists a neighborhood $U$ of $\hat{a}$ in $\textnormal{Ker}(d_F h(\hat{x}))$, a neighborhood $V$  of $\hat{b}$ in $E_1$ such that $U+V\subset \Omega$ and a unique continuous function  $\varphi : U\to V$  such that

$(\alpha)$ $\varphi(\hat{a})=\hat{b}$.

$(\beta)$ $\forall x\in U$, $h(x+\varphi(x))=0$.

$(\gamma)$ $\varphi$ is  Fr\'echet  differentiable at $\hat{a}$, and $d_F \varphi(\hat{a})=0$.
\vskip5mm
Let us define  $g_i: U\to \R$, for all $i\in \lbrace 0,...,m\rbrace$ by $g_i(x)=f_i(x+\varphi(x))$ for all $x\in U\subset \textnormal{Ker}(d_F h(\hat{x}))$.  Notice that $g_i(\hat{a})=f_i(\hat{x})$ for all $i\in \lbrace 0,...,m\rbrace$. By assumption and part $(\beta)$ above, we see that $\hat{a} \in U$ is a solution of the following problem in the Banach space $\textnormal{Ker}(d_F h(\hat{x}))$
\begin{equation*}
(\mathcal{H}_1)
\left \{
\begin{array}
[c]{l}
\max g_0\\
x\in U\\
\forall i \in \lbrace 1, ..., m \rbrace, g_{i}(x) \geq 0
\end{array}
\right. 
\end{equation*}
Now, using Theorem \ref{thmp}, we get some real numbers $\lambda^0\geq 0,..., \lambda^{m}\geq 0$ such that 

$(i')$ $(\lambda^0,..., \lambda^{m})\neq (0,...,0)$ 

$(ii')$ $\forall i \in \lbrace 1, ..., m \rbrace$, $\lambda^{i} g_{i}(\hat{a}) = 0$, that is, $\lambda^{i} f_{i}(\hat{x}) = 0$ and 

$(iii')$ $\sum_{i=0}^m  \lambda^{i} D^-_M g_{i}(\hat{a})(\mu) \leq 0, \hspace{1mm} \forall \mu\in \textnormal{Ker}(d_F h(\hat{x})).$
\vskip5mm
{\bf Claim.} $D^-_M g_i(\hat{a})(\mu)= D^-_M f_i(\hat{x})(\mu)$, for all $\mu \in \textnormal{Ker}(d_F h(\hat{x}))$ and for all $ i \in \lbrace 0, ..., m \rbrace$.

\begin{proof}[Proof of the claim] For all $\mu \in \textnormal{Ker}(d_F h(\hat{x}))$, we have $\varphi(\hat{a}+t\mu)=\varphi(\hat{a}) + td_F \varphi(\hat{a}) (\mu) +o(t)=\varphi(\hat{a})+o(t)$ (where $\lim_{t\to 0}\frac{\|o(t)\|_E}{t}=0$). Thus, we have for all $i \in \lbrace 1, ..., m \rbrace$,
\begin{eqnarray*}
t^{-1} [g_i(\hat{a}+t\mu)-g_i(\hat{a})]&=& t^{-1} [f_i(\hat{a}+t\mu +\varphi(\hat{a}+t\mu))-f_i(\hat{x})]\\
&=& t^{-1}[f_i(\hat{a}+t\mu +\varphi(\hat{a}) +td_F \varphi(\hat{a}) (\mu) +o(t))-f_i(\hat{x})]\\
&=& t^{-1}[f_i(\hat{x}+t\mu +o(t))-f_i(\hat{x})]\\
&=& t^{-1}[f_i(\hat{x}+t\mu +o(t))-f_i(\hat{x}+t\mu)] + t^{-1}[f_i(\hat{x}+t\mu)-f_i(\hat{x})]
\end{eqnarray*}
Since $f_i$ is Lipschitz  in a neighborhood of  $\hat{x}$, there exists a constant $\textnormal{Lip}(f_i)\geq 0$ such that 
$$ t^{-1}|f_i(\hat{x}+t\mu +o(t))-f_i(\hat{x}+t\mu)|\leq \textnormal{Lip}(f_i) \frac{\|o(t)\|_E}{t}\to 0, \textnormal{ when } t\to 0^+.$$
Passing to the {\it "liminf"} in the above equalities, we get that $D^- g_i(\hat{a})(\mu)=D^- f_i(\hat{x})(\mu)$ for all $\mu \in \textnormal{Ker}(d_F h(\hat{x}))$. It follows that  $D^-_M g_i(\hat{a})(\mu)=D^-_M f_i(\hat{x})(\mu)$ for all $\mu \in \textnormal{Ker}(d_F h(\hat{x}))$. This ends the proof of the claim.
\end{proof}
We return to the proof of the theorem. Using the claim and  $(iii')$ we have 
\begin{eqnarray}\label{formula21}
\sum_{i=0}^m  \lambda^{i} D^-_M f_{i}(\hat{x}) (\mu) \leq 0, \hspace{1mm} \mu\in \textnormal{Ker}(d_F h(\hat{x})). 
\end{eqnarray}
Since the map $-\sum_{i=0}^m  \lambda^{i} D^-_M f_{i}(\hat{x})$ is (Lipschitz) continuous sublinear functional (see Proposition \ref{prop1}),  then there exists a continuous linear functional $e^*\in E^*_1$ such that $e^*(\nu)\leq -\sum_{i=0}^m  \lambda^{i} D^-_M f_{i}(\hat{x})(\nu)$ for all $\nu \in E_1$ (see \cite[p. 177-220]{NB}). Let us set $w^*_0=e^*\circ (d_F h(\hat{x})_{|E_1})^{-1}\in W^*$. Then, we have
\begin{eqnarray}\label{formula22}
(\sum_{i=0}^m  \lambda^{i} D^-_M f_{i}(\hat{x}))_{|E_1}  +w^*_0\circ d_F h(\hat{x})_{|E_1} \leq 0.
\end{eqnarray}
Since $ d_F h(\hat{x})\circ I_{\textnormal{Ker}(d_F h(\hat{x}))}=0$, by using the formulas (\ref{formula21}) and (\ref{formula22}), we obtain
\begin{eqnarray} \label{EF}
\sum_{i=0}^m  \lambda^{i} D^-_M f_{i}(\hat{x}) (u) +\langle w^*_0, d_F h(\hat{x})(u) \rangle \leq 0,\hspace{1mm} \forall u\in E=\textnormal{Ker}(d_F h(\hat{x}))\oplus E_1,
\end{eqnarray}
where, $(\lambda^0,..., \lambda^{m})\neq (0,...,0)$. If moreover we assume that there exists $w\in \textnormal{Ker}(d_F h(\hat{x}))$ such that, $D^-_M f_j(\hat{x})(w) > 0$, whenever $j\in \lbrace i \in \lbrace 1, ..., m \rbrace : f_i(\hat{x}) = 0 \rbrace$, then we can chose $\lambda_0=1$. Indeed, suppose by contradiction that $\lambda^0=0$ then we have $(\lambda^{1}, ..., \lambda^{m}) \neq (0, ..., 0)$. Since, $\lambda^i=0$ whenever $i\in \lbrace 1, ..., m \rbrace$ and $f_i(\hat{x})\neq 0$ by part $(ii')$, we get that $\sum_{i=1}^m  \lambda^{i} D^-_M f_{i}(\hat{x})(w) >0$, which contradicts  $(\ref{EF})$. Hence, $\lambda^0\neq 0$. Dividing  if necessary the multipliers $(\lambda^i)_{0\leq i \leq m}$ by $\lambda^0\neq 0$, we can assume that $\lambda^0=1$. This ends the proof.
\end{proof}
\begin{remark} In the above theorem, if we assume moreover that the functions $f_i$, $i \in \lbrace 1, ..., m \rbrace$, are Gateaux-differentiable, then the inequality in $(iii)$ becomes an equality.
\end{remark}
\section*{Acknowledgement}
This research has been conducted within the FP2M federation (CNRS FR 2036) and  SAMM Laboratory of the University Paris Panthéon-Sorbonne.

\bibliographystyle{amsplain}

\end{document}